\documentclass[12pt]{amsart}
\usepackage[margin=20mm]{geometry}
\usepackage{amsmath,amssymb}            
\usepackage{hyperref}

\newcommand{\N}{\mathbb{N}}

\DeclareMathOperator{\e}{\mathrm{e}}

\newcommand{\floor}[1]{\left\lfloor #1 \right\rfloor}
\newcommand{\abs}[1]{\left| #1 \right|}

\newcommand{\ld}{\ll_{_d}}
\newcommand{\rb}[1]{\left( #1 \right)}

\renewcommand{\mod}{\operatorname{mod}}
\newcommand{\Card}{\operatorname{Card}}

\newtheorem{theorem}{Theorem}
\newtheorem{proposition}{Proposition}

\newtheorem{lemma}{Lemma}

\allowdisplaybreaks

\begin{document}
\title{Randomness and non-randomness properties of Piatetski-Shapiro sequences modulo m}

\author{Jean-Marc Deshouillers}
\email{jean-marc.deshouillers@math.u-bordeaux.fr}
\address{Institut de Math\'ematiques de Bordeaux \\
Universit\'e de Bordeaux, CNRS, Bordeaux INP\\
351, cours de la Lib\'eration\\
33 405 Talence, France}
\author{Michael Drmota}
\email{michael.drmota@tuwien.ac.at}
\address{Institut f\"ur Diskrete Mathematik und Geometrie
TU Wien\\
Wiedner Hauptstr. 8--10\\
1040 Wien, Austria}
\author{Clemens M\"ullner}
\email{mullner@math.univ-lyon1.fr}
\address{
CNRS, Universit\'e de Lyon, Universit\'e Lyon 1, Institut Camille
Jordan, 43 boulevard du 11 novembre 1918, 69622 Villeurbanne Cedex, France
}

\author{Lukas Spiegelhofer}
\email{lukas.spiegelhofer@tuwien.ac.at}
\address{Institut f\"ur Diskrete Mathematik und Geometrie
TU Wien\\
Wiedner Hauptstr. 8--10\\
1040 Wien, Austria}
%



\date{\today} 
\thanks{This work was supported by the
Austrian Science Foundation FWF, SFB F5502-N26 ``Subsequences of Automatic Sequences and Uniform Distribution'', which is a part of the Special Research Program ``Quasi Monte Carlo Methods: Theory and Applications'',
by the joint ANR-FWF project  ANR-14-CE34-0009, I-1751 MuDeRa, Ci\^encia sem Fronteiras (project PVE 407308/2013-0) and by the European Research Council (ERC) under the European Union’s Horizon 2020 research and innovation programme under the Grant Agreement No 648132. }

\begin{abstract}
We study Piatetski-Shapiro sequences $(\lfloor n^c\rfloor)_n$ modulo m, for non-integer $c >1$ and positive $m$,
and we are particularly interested in subword occurrences in those sequences.
We prove that each block $\in\{0,1\}^k$ of length $k < c + 1$ occurs as a subword with the frequency $2^{-k}$, while there are always blocks that do not occur. In particular, those sequences are not normal.
For $1<c<2$, we estimate the number of subwords from above and below, yielding the fact that our sequences are deterministic and not morphic.
Finally, using the Daboussi-K\'{a}tai criterion, we prove that the sequence $\lfloor n^c\rfloor$ modulo m is asymptotically orthogonal to multiplicative functions bounded by $1$ and with mean value $0$.
\end{abstract}

\maketitle

\section{Introduction}

The purpose of this paper is to study properties of Piatetski-Shapiro sequences $(\lfloor n^c\rfloor)_n$ modulo $m$ for positive and non-integer $c> 1$.

We will show that the sequence $(x_n)_n$ where $x_n=\lfloor n^c\rfloor\bmod m$ has some quasi-random properties as well as properties similar to those of a determistic sequence.
For example $(x_n)_n$ is $k$-normal for $k\le c$ but not $k$-normal for all $k$.
On the other hand the sequence $(x_n)_n$ is asymptotically orthogonal to the M\"obius function as it is expected for determistic sequences.
We will be more precise on these statements in Section~\ref{sec_results}.

Piatetski-Shapiro sequences $(\lfloor n^c\rfloor)_n$ are very well studied sequences and are an active area of research.
They are named after I. Piatetski-Shapiro, who proved the following prime number theorem~\cite{P1953}:
if $1<c<\frac {12}{11}$, then
\begin{equation*}
\left\lvert\left\{n\leq x: \floor{n^c}\text{ is prime}\right\}\right\rvert\sim \frac{x}{c\log x}.
\end{equation*}
This asymptotic formula is now known for $1<c<\frac{2817}{2426}$ (see Rivat and Sargos~\cite{RS2001}), moreover, it is true for almost all $c\in(1,2)$ (see Leitmann and Wolke~\cite{LW1975}).
We also refer to the paper~\cite{BBBSW2009} by Baker et al., giving a collection of arithmetic results on Piatetski-Shapiro sequences.

A different line of research is given by $q$-\emph{multiplicative functions} $\varphi$ along Piatetski-Shapiro sequences.
These functions satisfy $\varphi(q^ka+b)=\varphi(q^ka)\varphi(b)$
for all $a,k\geq 0$ and $0\leq b<q^k$.
Mauduit and Rivat~\cite{MR2005} proved an asymptotic formula concerning $q$-multiplicative functions $\varphi:\mathbb N\rightarrow\{z:\lvert z\rvert=1\}$ along $\lfloor n^c\rfloor$, where $1<c<7/5$.
This contains in particular the result that the Thue--Morse sequence on $\{-1,+1\}$ (which is $2$-multiplicative) along $\lfloor n^c\rfloor$ attains each of its two values with asymptotic density $1/2$, as long as $c<1.4$.
M\"ullner and Spiegelhofer~\cite{MS2017} improved this bound to $1<c<1.5$, and very recently, Spiegelhofer~\cite{S2018} obtained the range $1<c<2$.
Moreover, Mauduit and Rivat's result was transferred to automatic sequences by Deshouillers, Drmota, and Morgenbesser~\cite{DDM2012}.

A more basic question concerns Piatetski-Shapiro sequences modulo $m$.
Rieger~\cite{R1967} proved an asymptotic expression for the number of $\lfloor n^c\rfloor$ that lie in a residue class modulo $m$,
a result that was sharpened by Deshouillers~\cite{D1973}.

Mauduit, Rivat and S\'ark\"ozy~\cite{MRS2002} studied pseudorandomness properties of $(\lfloor n^c\rfloor \bmod 2)_n$ (more precisely, of the sequence $(\lfloor 2n^c\rfloor\bmod 2)_n$).
They proved that the \emph{well distribution measure} and the \emph{correlation measure} of order $k$ of this sequence are both small; these properties are to be expected from a ``good'' pseudorandom sequence.

In the present paper, we continue the study of $(\lfloor n^c\rfloor\bmod m)_n$ and establish further randomness- and non-randomness properties of this sequence.

\section{Results}\label{sec_results}

Let $m$ and $k$ be positive integers; a sequence of integers $(u_n)_n$ is said to be $k$-uniformly distributed modulo $m$ if for every 
block $B \in \{0,1, \ldots, m-1\}^k$ 

\[
\lim_{N\to\infty} \frac 1N \Card \{n< N : (u_n,u_{n+1},\ldots,u_{n+k-1}) \equiv B (\operatorname{mod} m) \} = \frac 1{m^k};
\]
we equivalently say that the sequence $(u_n \,\operatorname{mod} m)_n$ is $k$-normal. We further say that $(u_n)_n$ is completely uniformly distributed modulo $m$ if it is $k$-uniformly distributed modulo $m$ for any $k$ or, equivalently, that $(u_n \operatorname{mod} m)_n$ is normal if it is $k$-normal for any $k$.\\

Our first result says that the Piatetski-Shapiro sequence modulo $m$ is $k$-uniformly distributed modulo $m$ up to some level in $k$. 

\begin{theorem}\label{thm_k_uniform}
Suppose that $c>1$ is not an integer and let $m$ be a positive integer. Then the sequence 
$(\lfloor n^c\rfloor)_n$ is $k$-uniformly distributed modulo $m$ for $1\le k\le c+1$.
\end{theorem}

However, this is no longer true for all $k$, even in a weaker sense. 
\begin{theorem}\label{thm_not_normal}
Let $m\ge 2$ be an integer and  let $c>1$ a real number which is not an integer. Then the sequence 
$x = (\lfloor n^c\rfloor \, \operatorname{mod} m)_n$ is not normal.
More precisely, there exist some $k$ and a block $B\in  \{0,1, \ldots, m-1\}^k$ which
does not appear in $x$.
\end{theorem}
Note that this behaviour is different from that of $(s_2(\lfloor n^c \rfloor) \operatorname{mod} 2)_n$ (the Thue-Morse sequence along $(\lfloor n^c \rfloor)_n $) since in this case we have normality for $1 < c < 3/2$ \cite{MS2017}.\\
Next, we discuss the case $1< c < 2$ in more detail. 
We recall that the subword complexity $L_k$, $k\ge 1$, of a sequence $u$ 
with values in $ \{0,1, \ldots, m-1\}$
is the 
number of different blocks 
$B\in  \{0,1, \ldots, m-1\}^k$
that appear as a contiguous subsequence of $u$. A sequence $u$ is said to be deterministic if its topological entropy $h$ of the corresponding dynamical system is zero, or in other terms
\[
h = \lim_{k\to \infty} \frac 1k \log L_k = 0.
\]
Among the deterministic sequences, a simple class is that of morphic sequences which are the coding 
of a fixed point of a substitution, see Allouche and Shallit~\cite{AS2003}; they satisfy 
\[
L_k \ll k^2
\]

The following result implies that  for $1<c<2$, the sequence $(\lfloor n^c \rfloor \, \operatorname{mod} m)_n$ is deterministic but not morphic.

\begin{theorem}\label{thm_complexity}
Assume that $1<c<2$ and let $m\geq 2$ be an integer.
There exists a constant $C_1$ such that the subword complexity $L_k$ of the sequence $\bigl(\lfloor n^c\rfloor\bmod m\bigr)_{n}$ is bounded above by $C_1k^r$, for all $r>\max\{4/(2-c),6\}$.

Moreover, there is a constant $C_2$ such that $L_k\geq C_2k^3$.
\end{theorem}

It is a famous conjecture by Sarnak \cite{Sarnak}
that every 
bounded
deterministic sequence $(u_n)_n$ is
asymptotically orthogonal to the M\"obius function $\mu$, which means that one has
\[
\sum_{n< N} \mu(n)\, u_n = o(N), \qquad (N\to\infty).
\]
This is true in the case when $u_n = \lfloor n^c \rfloor \bmod m$. We even have the following result
\begin{theorem}\label{thm_sarnak}
Suppose that $c>1$ is not an integer and that $m\ge 2$ is an integer. 
Let $G$ be a complex valued function defined on $\{0,1,\ldots, m-1\}$. 
Then, for every multiplicative
function $f(n)$ with $|f(n)|\le 1$ and the property 
\[
\sum_{n< N} f(n)= o(N), \qquad (N\to\infty),
\]
we have
\[
\sum_{n< N} f(n) G \left( \lfloor n^c \rfloor \bmod m \right) = o(N), \qquad (N\to\infty).
\]
\end{theorem}

In Section 3, we define some more notation and study the trigonometric sums and discrepancies relevant for our questions. Theorems 1, 2, 3 and 4 are respectivilely proved in the four subsequent sections.

\section{Notation, trigonometric sums, discrepancy}

\subsection{Notation}
For a real number $u$, we let $\e(u) = \exp(2\pi i u)$.

For a real number $c$ we use Knuth's notation for the falling factorials defined recursively by $c^{\underline{0}} = 1$ and $c^{\underline{k}} = c^{\underline{k-1}}(c-k+1)$.

For $k$ a positive integer and $h = (h_1, h_2, \ldots, h_k) \in \mathbb{R}^k$, we let $\|h\|_{\infty} = \max \{|h_1|, \ldots, |h_k|\}$.

For a real number $x$ we let $x = \lfloor x \rfloor + \{x\}$ be its only decomposition as a sum of an integer and an element in $[0, 1)$. We notice that the map $x \mapsto \{x\}$ permits to identify $\mathbb{T} = \mathbb{R} / \mathbb{Z}$ and the interval $[0, 1)$. We let $\|x\| = \min\{\{x\}, 1-\{x\}\}$ be the so-called distance of $x$ to {the} nearest ineger. For a positive integer $k$ we identify  $\mathbb{T} ^k$ and $[0, 1)^k$. An \emph{interval} $I$ in $[0, 1)^k$ is a cartesian product $\prod_{1 \le i \le k} [a_i, b_i)$, with $0 \le a_i \le b_i \le 1$ for $1 \le i \le k$; its Lebesgue measure $\prod_{1 \le i \le k} (b_i-a_i)$ is denoted by $\lambda(I)$. For an interval $I \subset [0, 1)^k$, we denote by $\chi_I$ its indicator (also called characteristic) function.

The \emph{discrepancy} of a finite set $X = \{x_1, x_2, \ldots, x_N\}$ of elements of $\mathbb{R}^k$ is defined by
\begin{equation}\label{defdiscr}
D_N(X) = D_N(x_1, x_2, \ldots, x_N) = \sup_{I \subset [0, 1)^k} 
\abs{ \frac{1}{N} \sum_{n=1}^N \chi_I(\{x_n\}) - \lambda(I)}.
\end{equation}

\subsection{Trigonometric sums over polynomials}
\begin{proposition}\label{lem_estimate}
Let $k \ge 1$ and $P(n) = \sum_{i=0}^k \alpha_in^{i}$ be a polynomial of degree $k$ with real coefficients. Let $q, R, h$ be positive integers and $p$ an integer such that
\begin{equation}\label{alphak}
\gcd(p, q) = 1, \;\abs{ \alpha_k - \frac{p}{q}} \le \frac{1}{q^2} \; \text{ and }\; 2hk!R^{2-1/2^{k-2}} \le q.
\end{equation}
For $N \ge R$ we have
\begin{equation}\label{eqnlem_estimate}
\frac{1}{N}\abs{\sum_{1 \le n \le N}\e(hP(n))} \ll_{_k} \left(\frac{1}{R} + \frac{q}{N}\right)^{1/2^{k-1}} .
\end{equation}
\end{proposition}
\begin{proof}
Our first step is to use, for $k \ge 2$, the Weyl-van der Corput method to reduce the evaluation of the left hand side of (\ref{eqnlem_estimate}) to the evalutation of geometric sums. We apply Lemma 2.7 of \cite{GK1991} with
\begin{equation}\label{prep2-7}
q=k-1, \; Q = 2^{k-1}, \; I=(0, N], \; H_j = R_j = R^{1/2^{k-1-j}} \text{ for } 1 \le j \le k-1,
\end{equation}
notice that the condition $R \le N = \abs{I}$ is fulfilled and get
\begin{equation}\label{reducGsum}
\frac{1}{N}\abs{\sum_{1 \le n \le N}\e(hP(n))} \ll \left(\frac{1}{R} + \frac{1}{NR_1\cdots R_{k-1}}      \sum_{\substack{1 \le r_1 \le R_1\\ \ldots\ldots \\ 1 \le r_{k-1} \le R_{k-1}}}\abs{\sum_ {n=1}^{N-r_1-\cdots-r_{k-1}}\e(\alpha_k hk! r_1 \cdots r_{k-1}n)}\right)^{1/2^{k-1}}.
\end{equation}
Let now $\ell = h k! r_1\cdots r_{k-1}$; for any $M$ we have
$$
\abs{\sum_{n=1}^M \e(\alpha_k \ell n)}  \le \frac{2}{\abs{\e(\alpha_k \ell) - 1}} = \frac{1}{\abs{\sin(\pi\alpha_k \ell)}}. 
$$
We are thus looking for a lower bound for $\abs{\sin(\pi\alpha_k \ell)}$. We  have
$$
\abs{\sin(\pi\alpha_k \ell)} = \abs{\sin\left(\pi\frac{\ell p}{q} + \pi \ell\left(\alpha_k - \frac{p}{q}\right)\right)}.
$$
Relation (\ref{alphak}) implies $1\le  \ell \le q-1$ (in particular $q \ge 2$) and $\gcd(p, q)=1$: the number $\ell p$ is never $0$ modulo $q$. On the other hand, we have
$$
\abs{\ell\left(\alpha_k - \frac{p}{q}\right)} \le \frac{\ell}{q^2} \le \frac{1}{2q}.
$$
The last two relations imply
$$
\abs{\sin(\pi\alpha_k \ell)} \ge \sin\left(\frac{\pi}{2q}\right) \ge \frac{1}{q} 
$$
and so for any $M$ one has  $\abs{\sum_{n=1}^M \e(\alpha_k \ell n)}  \le q$. This easily implies the validity of Proposition \ref{lem_estimate} when $k \ge 2$.\\
The case when $k=1$ is now straightforward. We have
$$
\frac{1}{N}\abs{\sum_{1 \le n \le N}\e(hP(n))} =\frac{1}{N}\abs{\sum_{1 \le n \le N}\e(\alpha_k h n)} .
$$
Relation (\ref{alphak}) implies $2h \le q$ and the previous reasoning implies
$$
 \frac{1}{N}\abs{\sum_{1 \le n \le N}\e(hP(n))} \ll \frac{q}{N},
$$
and (\ref{eqnlem_estimate}) is satisfied for any value of $R$.
\end{proof}

\subsection{Trigonometric sums involving the function $n^c$ at consecutive arguments}
\begin{proposition}\label{lem_3}
Let $c>1$ be a non integral real number, let $L=\lfloor c \rfloor +1$ and let $m$ be a positive integer. For any $L$-tuple $h = (h_0, h_1, \ldots, h_{L-1})$ of integers which are not all $0$ and any positive integer $N$, we have
\begin{equation}\notag
\sum_{n=N}^{2N-1} \e\left(\frac{1}{m}\sum_{\ell=0}^{L-1} h_{\ell}(n + \ell)^c\right) \ll_{L, c, m} \|h\|_{\infty} \, N^{1-\frac{\|c\|}{2^{(c+1)}}},
\end{equation}
as soon as
\begin{equation}\notag
 \|h\|_{\infty}  =o\left(N^{1-\{c\}}\right).
\end{equation}
\end{proposition}

\begin{proof}
In order to apply classical upper bounds for trigonometrical sums, we need to have a lower and an upper bound for the absolute value of some derivative of the function $f$ defined by $f(x)=  \frac{1}{m}\sum_{\ell=0}^{L-1} h_{\ell}(x + \ell)^c$.\\
\indent If $\frac{1}{m}\sum_{\ell=0}^{L-1} h_{\ell} \neq 0$, then for $x \in [N, 2N-1]$ and any integer $k \le c+1$ we have
\begin{equation}\notag
N^{c-k} \ll_{L, c, m} \mid f^{(k)}(x) \mid \ll_{L, c , m} \|h\|_{\infty} N^{c-k},
\end{equation}
which is fine for our purpose.\\
\indent But if $\sum_{\ell=0}^{L-1} h_{\ell} = 0$, the order of magnitude of $f^{(k)}$ is no longer $N^{c-k}$. In that case, we use the Taylor expansion for $(x+\ell)^c$; the next term is now $\frac{c}{m}\sum_{\ell=0}^{L-1} h_{\ell} \ell (x + \ell)^{c-1}$. If $ \sum_{\ell=0}^{L-1} h_{\ell} \ell \neq 0$, we have $N^{c-1-k} \ll_{L, c, m} \mid f^{(k)}(x) \mid \ll_{L, c, m}\|h\|_{\infty} N^{c-1-k}$ and we are done ; if $ \sum_{\ell=0}^{L-1} h_{\ell} \ell =0$, we go to the next term in the Taylor expansion, and so on...\\
\indent However, since the vector $h$ is non zero, there exists $r \in [0, L-1]$ such that
\begin{equation}\label{rOK}
\sum_{\ell=0}^{L-1} h_{\ell} \ell^r \neq 0:
\end{equation}
if it were not the case, we would have $Ah=0$, where $A = \left(j^{i}\right)_{0 \le i, j \le L-1}$  is the transposed matrix of an invertible Vandermonde matrix and $h$ a non zero vector, which is not possible.\\

Let $r$ be the smallest non-negative integer for which (\ref{rOK}) holds. For $k \le c + 1$, we have
\begin{equation}\label{deriv}
N^{c-r-k} \ll_{L, c, m} \mid f^{(k)}(x) \mid \ll_{L, c , m} \|h\|_{\infty} N^{c-r-k}.
\end{equation}

Let us first assume that $0 \le r \le L-3$, a case which may occur only when $c>2$. We let
$$
q = \lfloor c \rfloor - r - 1.
$$
We notice that 
$$q \ge L -1- (L-3)-1 = 1$$ and that 
$$c-r-(q+2) = c - r - \lfloor c \rfloor +r +1 - 2 = c - \lfloor c \rfloor - 1 = \{c\}-1.$$
For $x \in [N, 2N-1]$, we have
$$
N^{\{c\}-1}  \ll_{L, c, m}    \mid f^{(q+2)}(x) \mid    \ll_{L, c, m}  \|h\|_{\infty} \, N^{\{c\}-1}.
$$
We let 
$$
\lambda = \min_{x \in [N, 2N-1]}  \mid f^{(q+2)}(x) \mid \;  \text{ and }\;  \alpha \lambda = \max_{x \in [N, 2N-1]}  \mid f^{(q+2)}(x) \mid.
$$
The previous double inequality implies that there are constants $\kappa_1$ and $\kappa_2$ depending at most on $L, c, m$ such that
$$
\lambda = \kappa_1 N^{\{c\}-1}   \; \text{ and } \; 1 \le \alpha \le \kappa_2 \|h\|_{\infty}. 
$$
Theorem 2.8 of \cite{GK1991} implies that we have, with $Q = 2^q$,
\begin{eqnarray}\notag
\sum_{n=N}^{2N-1} \e\left(\frac{1}{m}\sum_{\ell=0}^{L-1} h_{\ell}(n + \ell)^c\right) & \ll_{L, c, m} & N\left(\|h\|_{\infty}^2 N^{\{c\}-1}\right)^{1/(4Q-2)} + \\
\notag & & N^{1-1/2Q}\|h\|_{\infty}^{1/2Q} + N^{1-2/Q+1/Q^2}N^{(\{c\}-1)/2Q}.
\end{eqnarray}
Using the inequalities
$$
1 \le Q \le 2Q \le 4Q-2 \le 4\times 2^{c-1}-2 \le 2^{c+1}-2 \; \text{ and }\; \{c\}-1 \ge -\|c\|,
$$
one obtains Proposition \ref{lem_3}.\\

Let us now assume that $r = L-2 = \lfloor c \rfloor -1$. In this case, we have $c-r-2=c-(\lfloor c \rfloor -1)-2 = \{c\} -1$, and, thanks to (\ref{deriv}), for $x \in [N, 2N-1]$, we have
$$
N^{\{c\}-1}  \ll_{L, c, m}    \mid f^{(2)}(x) \mid    \ll_{L, c, m}  \|h\|_{\infty} \, N^{\{c\}-1}.
$$
We let 
$$
\lambda = \min_{x \in [N, 2N-1]}  \mid f^{(2)}(x) \mid \;  \text{ and }\;  \alpha \lambda = \max_{x \in [N, 2N-1]}  \mid f^{(2)}(x) \mid.
$$
The previous double inequality implies that there are constants $\kappa_1$ and $\kappa_2$ depending at most on $L, c, m$ such that
$$
\lambda = \kappa_1 N^{\{c\}-1}   \; \text{ and } \; 1 \le \alpha \le \kappa_2 \|h\|_{\infty}. 
$$
Theorem 2.2 of \cite{GK1991} implies that we have
\begin{equation}\notag
\sum_{n=N}^{2N-1} \e\left(\frac{1}{m}\sum_{\ell=0}^{L-1} h_{\ell}(n + \ell)^c\right)  \ll_{L, c, m}  N\|h\|_{\infty} N^{(\{c\}-1)/2}+  N^{(1-\{c\})/2},
\end{equation}
in which case, Proposition \ref{lem_3} is satisfied.\\

We now consider the last case, when $r=L-1$. In this case, we have $c-r-1=c-(\lfloor c \rfloor )-1=\{c\}-1$ and so 
$$
N^{\{c\}-1}\ll_{L, c, m} \mid f'(x)\mid  \ll_{L, c, m}  \|h\|_{\infty} \, N^{\{c\}-1}.
$$
Since, by hypothesis, the last term is $o(1)$, we can apply The Kusmin-Landau lemma (Theorem 2.1 of \cite{GK1991}) and obtain
\begin{equation}\notag
\sum_{n=N}^{2N-1} \e\left(\frac{1}{m}\sum_{\ell=0}^{L-1} h_{\ell}(n + \ell)^c\right)  \ll_{L, c, m}  N^{1-\{c\}},
\end{equation}
in which case, Proposition \ref{lem_3} is again satisfied.\\
\end{proof}

\subsection{Discrepancy of a perturbed sequence}

We will make use of the following elementary property
\begin{proposition}\label{perturbed}
Let $N \ge 1$ , $\delta \ge 0$ and $x_1, x_2, \ldots, x_N$ and $y_1, y_2, \ldots, y_N$ be two families of real numbers such that for all $n \le N$ we have $\mid y_n - x_n \mid \le \delta$. Then we have
$$
D_N(y_1, y_2, \ldots, y_N) \le 2 D_N(x_1, x_2, \ldots, x_N)  + 2 \delta.
$$
\end{proposition}
\begin{proof}
Let $I = [a, b) \subset [0, 1)$. We let $I_{\delta}^+ = ([a-\delta, b+ \delta) + \mathbb{Z})\cap [0, 1)$ and notice that  $I_{\delta}^+$ is either an interval or the union of two intervals; we let $I_{\delta}^-$  to be the interval  $[a+\delta, b- \delta)$ if $b-a > 2 \delta$ or the empty set otherwise. We have
$$
I_{\delta}^- \subset I \subset I_{\delta}^+  \; ,\; \lambda(I_{\delta}^+ ) \le \lambda(I)+ 2 \delta\; \text{ and } \; \lambda(I_{\delta}^- ) \ge \lambda(I)- 2 \delta.
$$
We thus have
$$
\sum_{1 \le n \le N} \chi_I(y_n)- \lambda(I) \le \sum_{1\le n \le N} \chi_{I_{\delta^+}}(x_n) - \lambda(I) \le  \lambda(I_{\delta}^+) + 2 D_N(x_1, \ldots, x_N)- \lambda(I) \le 2 D_N(x_1, \ldots, x_N) + 2 \delta
$$
and
$$
\sum_{1 \le n \le N} \chi_I(y_n)- \lambda(I) \ge \sum_{1\le n \le N} \chi_{I_{\delta^-}}(x_n) - \lambda(I) \ge  \lambda(I_{\delta}^-) -  D_N(x_1, \ldots, x_N)- \lambda(I) \ge - D_N(x_1, \ldots, x_N) - 2 \delta,
$$
which implies
$$
D_N(y_1, y_2, \ldots, y_N) \le 2 D_N(x_1, \ldots, x_N) + 2 \delta.
$$
\end{proof}

\subsection{The multidimensinal Erd\H{o}s-Tur\'{a}n theorem}
For the case $k=1$, Erd\H{o}s and Tur\'{a}n \cite{ET} gave an upper bound for the discrepancy in terms of exponential sums. Their result has been generalised by Koksma \cite{Koksma} and Sz\"{u}sz \cite{Szusz} in the multidimensional case. The version we give is taken from \cite{DT} (Theorem 1.21, page 15).

\begin{proposition}\label{ETKS}
Let $X = \{x_1, x_2, \ldots, x_N\}$ be a finite set of elements of $\mathbb{R}^k$ and $H$ an arbitrary positive integer. We have
\begin{equation}\label{equETKS}
D_N(X)\le \left(\frac{3}{2}\right)^k\left(\frac{2}{H+1} + \sum_{0<\|h\|_{\infty}\le H} \frac{1}{r(h)}\abs{\frac{1}{N} \sum_{n=1}^N \e(h\cdot x_n)}\right),
\end{equation}
where, for $h = (h_1, h_2, \ldots, h_k) \in \mathbb{Z}^k$, we let $r(h) = \prod_{i=1}^k \max\{1, |h_i|\}$ and $u\cdot v$ denote the usual scalar product of two elements $u$ and $v$ in $\mathbb{R}^k$.
\end{proposition}

\subsection{Discrepancy of a polynomial sequence}
We give here an upper bound for the discrepancy of a polynomial sequence in terms of rational approximations of its coefficients. This will be useful for the proof of Theorem \ref{thm_not_normal}.

\begin{proposition}\label{prp_approx_k}
  Let $P(x) = \sum_{i=0}^{d} \alpha_i x^{i}$ be a  polynomial of degree $d$ with real coefficients. For $i \in [1, d]$, we let $Q_i$, $q_i$ and $p_i$ be rational integers such that
\begin{equation}\label{approxP}
\gcd(p_i, q_i) = 1 \; , \; 1 \le q_i \le Q_i \; \text{ and }\; \abs{ \alpha_i- \frac{p_i}{q_i}} \le \frac{1}{q_i\cdot Q_i} .
\end{equation}
  Then we have for any $k$ such that $1\leq k \le d$
  \begin{align}\label{eqn_coeff_d}
    D_N(P(1),\ldots, P(N)) \ll_{_d} q^{k} \log(eq_{k}) \rb{\frac{1}{q_{k}^{1/2}} + \frac{q_{k}}{N}}^{1/2^{(k-1)}} + 
          q_{k}^{1/2^k} \sum_{i = k+1}^{d} \frac{N^i}{Q_i},
  \end{align}
  where $q:= \prod_{i=k+1}^{d} q_i$.
\end{proposition}

\begin{proof}
  We want to separate the contribution of the different $\alpha_{k}$'s to the discrepancy. Relation (\ref{eqn_coeff_d}) is trivially true when $q_k= 1$ and we may assume that $q_k \ge 2$; since $p_k$ and $q_k$ are coprime, $p_k$ is different from $0$ and so is $\alpha_k$.
  We approximate the higher degree coefficients by rational numbers.
We define 
  \begin{align*}
    y_n &:= \sum_{i=0}^{k} \alpha_i n^{i} + \sum_{i=k+1}^{d} \frac{p_i}{q_i} n^{i}\\
    z_n &:= \sum_{i=k+1}^{d} \rb{\alpha_i - \frac{p_i}{q_i}} n^{i},
  \end{align*}
  with the usual convention that $z_n=0$ when $k = d$.\\
In order to apply the original one dimension Erd\H{o}s-Tur\'{a}n inequality, we have to estimate trigonometrical sums. We have
  \begin{align*}        
    \sum_{0 <h\leq H} \frac{1}{h} \abs{\frac{1}{N} \sum_{n\leq N} \e(h P(n))} &= \sum_{0 <h\leq H} \frac{1}{h} \abs{\frac{1}{N} \sum_{n\leq N} \e(h (y_n + z_n))}\\
        &= \sum_{0<h\leq H} \frac{1}{h} \abs{\frac{1}{N} \sum_{n\leq N} \e(h y_n) + \frac{1}{N} \sum_{n\leq N} \e(h y_n) \rb{\e(h z_n) -1}}\\
        &\leq \sum_{0<h\leq H} \frac{1}{h} \rb{\abs{\frac{1}{N} \sum_{n\leq N} \e(h y_n)} + \frac{1}{N} \sum_{n\leq N} \abs{\e(h y_n) \rb{\e(h z_n) -1} } }\\
        &\leq \sum_{0<h\leq H} \frac{1}{h} \abs{\frac{1}{N} \sum_{n\leq N} \e(h y_n)} +\sum_{0<h\leq H} \frac{1}{h}  \frac{1}{N} \sum_{n\leq N} 2 \pi \abs{h z_n}.
  \end{align*}
The last sum is easily treated thanks to~\eqref{approxP}. We have
\begin{equation}\notag
\sum_{0<h\leq H} \frac{1}{h}  \frac{1}{N} \sum_{n\leq N} 2 \pi \abs{h z_n} \le 2\pi H \sum_{i=k+1}^d \frac{N^{i}}{q_i \cdot Q_i}  \le 2\pi H \sum_{i=k+1}^d \frac{N^{i}}{Q_i}.
\end{equation}
We thus have
\begin{equation}\label{zn}
\sum_{0 <h\leq H} \frac{1}{h} \abs{\frac{1}{N} \sum_{n\leq N} \e(h P(n))}  \leq \sum_{0<h\leq H} \frac{1}{h} \abs{\frac{1}{N} \sum_{n\leq N} \e(h y_n)} + 2\pi H \sum_{i=k+1}^d \frac{N^{i}}{Q_i}.
\end{equation}

  Thus, we want to estimate $\abs{\frac{1}{N} \sum_{n\leq N} \e(h y_n)}$. The following lemma will permit us to reduce the question to the evaluation of trigonometrical sums over polynomials of degree $k$.
  
  \begin{lemma}\label{reducetok}
  With the above notation, for any integer $r$, there exists a polynomial $Q_r$ of degree $k$ with leading coefficient $\alpha_k$ such that
  \begin{equation}\notag
  y_{tq+r}-q^kQ_r(t) \in \mathbb{Z} \; \text{ for any }\; t \in \mathbb{Z}.
  \end{equation}
  \end{lemma}
  \begin{proof}
  We recall that $q =  \prod_{i=k+1}^{d} q_i$. This implies that for any integer $t$ the sum
$ \sum_{i=k+1}^{d} \frac{p_i}{q_i} (tq+r)^{i} $ is equal, up to an integer, to $ \sum_{i=k+1}^{d} \frac{p_i}{q_i} r^{i}$ which is independent of $t$. By the binomial expansion, the first part of $  y_{tq+r}$, namely  $\sum_{i=0}^{k} \alpha_i (qt+r)^{i}$, is easily seen to be a polynomial of degree $k$ with leading coefficient $\alpha_k q^k$.
 \end{proof}

  We have
  \begin{align*}
    \abs{\frac{1}{N} \sum_{n\leq N} \e(h y_n)}
      &= \abs{\frac{1}{N} \sum_{r\leq q} \sum_{\substack{n \leq N\\ n \equiv r \bmod q}} \e(h y_n)}
      \leq \sum_{0 \le r < q} \frac{1}{N} \abs{\sum_{\substack{n \leq N \\ n \equiv r \bmod q}} \e(h y_n)}\\
      &\le \sum_{0 \le r < q} \frac{1}{N} \abs{\sum_{0 \le t\le (N-r)/q} \e(h y_{tq+r}) }
     \le \sum_{0 \le r < q} \frac{1}{N} \abs{\sum_{0 \le t\le (N-r)/q} \e(h q^k Q_r(t))}. 
  \end{align*}

  We define the integers $R$ and $H$ by
  $$
  R= \lfloor q_k^{1/2}\rfloor \; \text{ and } \; H = \lfloor q_k^{(1/2^k)}/(2 k! q^k)\rfloor.
  $$
  
We may assume that $H \ge 1$, since otherwise Proposition \ref{prp_approx_k} is trivial. We readily check that  the condition $2 H q^{k} k! R^{2-1/2^{k-2}}\leq q_{k}$ holds and that as soon as $N$ is large enough we have $R \le (N-q)/q$. We can thus apply Proposition~\ref{lem_estimate} which implies that for $1 \le h \le H$ we have 
  \begin{equation}\notag
  \abs{\sum_{0 \le t\le (N-r)/q} \e(h q^k Q_r(t)} \ll ((N-r)/q) \left(\frac{1}{R} + \frac{q_k}{(N-r)/q}\right)^{1/2^{k-1}} \ll N \left(\frac{1}{R} + \frac{q_k}{N}\right)^{1/2^{k-1}}.
  \end{equation}
 This leads to
  $$
\sum_{0<h\leq H} \frac{1}{h} \abs{\frac{1}{N} \sum_{n\leq N} \e(h y_n)} \ll q \log(eH) \left(\frac{1}{R} + \frac{q_k}{N}\right)^{1/2^{k-1}} \ll q \log(eq_k) \left(\frac{1}{q_k^{1/2}} + \frac{q_k}{N}\right)^{1/2^{k-1}}.
  $$
  We combine this, Proposition \ref{ETKS} and (\ref{zn}), getting Proposition \ref{prp_approx_k}.
\end{proof}

We notice that the optimal choice of $H$ and $R$ permits to replace the term $q^k$ in (\ref{eqn_coeff_d}) by $q^{f(k)}$ where $f(k)$ tends to $1$ as $k$ tends to infinity, but this is irrelevant for our application. 

\section{Proof of Theorem~\ref{thm_k_uniform}}

To show that the sequence $(\lfloor n^c \rfloor)_n$ is $k$-uniformly distributed modulo $m$, it is enough to show that for any $B = (b_0, b_1, \ldots, b_{k-1}) \in \{0,1, \ldots, m-1\}^k$ we have, as $N$ tends to infinity
$$
\frac 1N \Card \{n \in [N, 2N) : (\lfloor n^c \rfloor,\lfloor (n+1)^c \rfloor,\ldots,\lfloor (n+k-1)^c \rfloor) = B \} - \frac 1{m^k} \; \text{ tends to }\; 0.
$$
Thanks to the straightforward equivalence
$$
\lfloor n^c \rfloor  \equiv b \, (\operatorname{mod} m)\;  \Longleftrightarrow \; \left\{\frac{n^c}{m}\right\} \in \left[\frac{b}{m}, \frac{b+1}{m}\right)
$$
and the definition of the discrepancy given above, we have
\begin{eqnarray}\label{repthrudisc}
\abs{\frac{1}{N} \Card\{n \in [N, 2N) \colon (\lfloor n^c \rfloor,\lfloor (n+1)^c \rfloor,\ldots,\lfloor (n+k-1)^c \rfloor) = B\} - \frac{1}{m^k}}\\
\notag \le D_N\left(\left\{\left(\frac{n^c}{m}, \frac{(n+1)^c}{m},\ldots,\frac{(n+k-1)^c}{m}\right) \colon n \in [N, 2N)\right\}\right).
\end{eqnarray}
To evaluate the right hand side of (\ref{repthrudisc}), we use Proposition \ref{ETKS} with $H= N^{\frac{\|c\|}{(k+2)2^{c+1}}}$. Combining it with Proposition \ref{lem_3}, we obtain
\begin{equation}\notag
\abs{\frac{1}{N} \Card\{n \in [N, 2N) \colon (\lfloor n^c \rfloor),\lfloor (n+1)^c \rfloor),\ldots,\lfloor (n+k-1)^c \rfloor)) = B\} - \frac{1}{m^k}} \ll_{k, c, m} N^{-\frac{\|c\|}{(k+2)2^{c+1}}}.
\end{equation}
Theorem \ref{thm_k_uniform} is thus proved.
\begin{flushright} $\Box$\end{flushright}

\section{Proof of Theorem~\ref{thm_not_normal}}

\subsection{Coefficients of polynomials with large discrepancy}

Our first step is to show that if a sequence which is close to a polynomial has a large discrepancy, the non constant coefficients of the underlying polynomial have very good approximations with rationals with bounded denominators.

\begin{theorem}\label{thm_pol_small_coeffs}
Let $\delta$ be a positive number, $d$ a natural integer. There exists a positive integer $M(\delta, d)$ having the following property:\\
Let $P(x) = \sum_{k=1}^d \alpha_k x^k$ be a polynomial of degree $d$ such that for any sufficiently large $N$, for any $\eta = (\eta_1, \ldots, \eta_N)$ with $\|\eta\|_{\infty} \le \delta$, we have
\begin{equation}\label{Ppertur}
D_N(P(1)+\eta_1, P(2) + \eta_2, \ldots, P(N)+\eta_N) > 4 \delta.
\end{equation}
Then for all sufficiently large $N$, we have
\begin{equation}\label{eqn_not-normal}
\forall k \in [1, d], \exists (p_k, q_k) \text{  with } \gcd(p_k, q_k) = 1, 1\le q_k \le M(\delta, d) \text{ and } \abs{\alpha_k - \frac{p_k}{q_k}} \le N^{-\frac{(k+2)!}{2(d+2)!}}.
\end{equation}
\end{theorem}

\begin{proof}
The perturbation by $\eta$ will be useful for the application, but we can easily cope with it: by Proposition \ref{perturbed}, the bounds $\|\eta\|_{\infty} \le \delta$ and (\ref{Ppertur}) imply
\begin{equation}\label{P}
D_N(P(1), P(2), \ldots, P(N)) >  \delta,
\end{equation}
which is the condition we are going to use from now on.\\

Let $N$ be a sufficiently large integer. For $i\in [1, d]$ we define
$$
N_i =\lfloor N^{\frac{(i+2)!}{2(d+2)!} } + 1\rfloor\; , \; Q_i  = \lfloor N_i^{1-\varepsilon} + 1\rfloor, \;\text{ where } \; \varepsilon = \frac{1}{2(d+2)},
$$
and we let $(p_i, q_i)$ be such that
\begin{equation}\notag
\gcd(p_i, q_i) = 1\; , \; 1 \le q_i \le Q_i \; \text{ and } \abs{\alpha_i - \frac{p_i}{q_i}} \le \frac{1}{q_i\cdot Q_i}.
\end{equation}

In order to prove the theorem, we shall show that for any $k$ in $[1, d]$, there exists $M_k(\delta, d)$ such that
\begin{equation}\label{induc}
\forall i \in [k, d], \exists (p_i, q_i) \text{  with } \gcd(p_i, q_i) = 1, 1\le q_i \le M_k(\delta, d) \text{ and } \abs{\alpha_i - \frac{p_i}{q_i}} \le N^{-\frac{(i+2)!}{2(d+2)!}}.
\end{equation}

We first prove (\ref{induc})  when $k=d$. By (\ref{P}) and Proposition \ref{prp_approx_k}, we have
\begin{equation}\notag
\delta \le D_{N_d}(P(1), \ldots, P(N_d)) \ld \log(eq_d) \rb{\frac{1}{q_{d}^{1/2}} + \frac{q_{d}}{N_d}}^{1/2^{d-1}} \ld \log(eq_d) \max\left(\frac{1}{q_{d}^{1/2d}}, \rb{\frac{Q_d}{N_d}}^{1/2^{d-1}} \right).
\end{equation}
By definition,  we have $Q_d \le 2N_d^{1-\varepsilon}$ and so the quantity $\log(eq_d) (Q_d/N_d)^{1/2^{d-1}}$ wich is less than $\log(eQ_d) (Q_d/N_d)^{1/2^{d-1}}$  tends to zero as $N_d$ tends to infinity and thus as $N$ tends to infinity; when $N$ is large enough, the term $\log(eq_d) q_d^{-1/2^d}$ has to be bounded from below, i.e. $q_d$ has to be bounded from above. This is the case $k=d$ of the theorem. \\

Assume now that~\eqref{induc} is proved for some $k \in [2, d]$ and let us show that it is also true for $k-1$. We are going to use Proposition \ref{prp_approx_k}, with $N = N_{k-1}$ and start with some preliminary computation. 
$$
q = \prod_{i=k}^d q_i \le M_k(\delta, d)^d.
$$
We also have
 \begin{align*}
    Q_{k-1}^{1/2^{k-2}} \sum_{i=k}^{d} \frac{N_k^i}{Q_i} &\leq \sum_{i=k}^{d} \frac{N_{k-1}^{i+1}}{Q_i}= \sum_{i=k}^{d} \frac{N^{(i+1)\frac{(k+1)!}{(d+2)!} } }{N^{(1 - \varepsilon)\frac{(i+2)!}{(d+2)!} } }= \sum_{i=k}^{d} \rb{\frac{N^{i+1}}{N^{(1 - \varepsilon)\frac{(i+2)!}{(k+1)!} } }}^{\frac{(k+1)!}{(d+2)!}}\\
        &\leq \sum_{i=k}^{d} \rb{\frac{N^{i+1}}{N^{(1-\varepsilon) (i+2)} } }^{\frac{(k+1)!}{(d+2)!}}= \sum_{i=k}^{d} \rb{\frac{1}{N^{1-\varepsilon(i+2)} } }^{\frac{(k+1)!}{(d+2)!}}\leq (d-k+1) N^{-\frac{1-(d+2)\varepsilon}{(d+2)}}.
  \end{align*}
We recall that $\varepsilon = 1/(2(d+2))$ and obtain
$$
q_{k-1}^{1/2^{k-2}} \sum_{i=k}^{d} \frac{N_{k-1}^i}{Q_i} \le Q_{k-1}^{1/2^{k-2}} \sum_{i=k}^{d} \frac{N_{k-1}^i}{Q_i} \leq d N^{-\frac{1}{2(d+2)}}.
$$
From that computation and Proposition  \ref{prp_approx_k}, we get
\begin{align*}
    \delta \le D_{N_{k-1}}(x_1,\ldots,x_{N_{k-1}}) &\ld M_k(\delta, d)^{d} \log(eq_{k-1}) \max \left(\frac{1}{q_{k-1}^{1/2^{k-1}}}\, , \left(\frac{q_{k-1}}{N_{k-1}}\right)^{1/2^{k-2}}\right)+ q_{k-1}^{1/2^{k-2}} \sum_{i=k}^{d} \frac{N_{k-1}^i}{Q_i}\\
   &\ld  M_k(\delta, d)^{d} \log(eq_{k-1}) \max \left(q_{k-1}^{-1/2^{k-1}}\, , N_{k-1}^{-\varepsilon/2{k-2}}\right) + N^{-\frac{1}{2(d+2)}}.
  \end{align*}
Arguing as above, we see that for $N$ large enough, this relation can hold only if $q_{k-1}$ is bounded from below by $M'(\delta, d)$, say. We let $M_{k-1}(\delta, d) = \max(M'(\delta, d), M_{k}(\delta, d))$.

Induction implies the validity of Theorem \ref{thm_pol_small_coeffs} with $M(\delta, d)  = M_1(\delta, d) $.
\end{proof}

\subsection{Non-uniformity modulo $m\ge3$ of perturbed polynomials}

The following result shows that the sequence of the integral parts of the values of a slightly perturbed real polynomial cannnot be uniform modulo any $m \ge 3$.

 \begin{theorem}\label{mmorethan2}
Let $m$ and $d$ be positive integers with $m \ge 3$. There exist a positive integer $N = N(m,d)$ and a block $B \in \{0, 1, \ldots m-1\}^N$ such that for any real polynomial $P$ of degree $d$ and any sequence $\eta = (\eta_n)_n$ of real numbers bounded by $1/20$ there exists $n \in [1, N]$ such that 
\begin{equation}\label{notB}
\lfloor P(n) + \eta_n\rfloor \not\equiv b_n\, (\mod m).
\end{equation}
\end{theorem}
\begin{proof}
We let $M = M(1/(20m), d)$ (where $M(.,.)$ was defined in Theorem~\ref{thm_pol_small_coeffs}) and $N \ge (40mdM!)^{(d+3)!}$ be integers satisfying  Theorem \ref{thm_pol_small_coeffs} (with $d=d$ and $\delta =1/(20m)$) and define the block $B$ to consist of $M!$ digits $2$ followed by $M!$ digits $1$ followed by $N-2M!$ digits $0$.

We assume that there exists a polynomial $P$ for which (\ref{notB}) does not hold; in particular, for any $n \in [2M!+1, N]$ we have $\lfloor P(n) + \eta_n\rfloor \equiv 0 \,(\mod m)$. We let $R(x) = P(x)/m$ and $\beta_n = \eta_n/m$. Since $N \ge 8M!$ the discrepancy of the sequence $(R(n) + \beta_n)_{1 \le n\le N}$ is larger than $1/2$. We can thus apply Theorem \ref{thm_pol_small_coeffs} with $\delta = 1/(20m)$. Let us write $R(x) = \sum_{k=1}^{d} \alpha_k x^k$; for any $k$ there exist coprime rational integers $p_k$, $q_k$ with $1 \le q_k \le M$ and
$$
\abs{\alpha_k - \frac{p_k}{q_k}} \le N^{-\frac{(k+2)!}{2(d+2)!}} \le \frac{1}{20md\cdot (2M!)^{d}},
$$
where the last inequality comes from the choice of $N$.

For $k \in [1, d]$ we let $\varepsilon_k = \alpha_k -p_k/q_k$ and $r(x) = \alpha_0 + \sum_{k=1}^d \varepsilon_k (M!)^k x^k$. Since $M!/q_k \in \N$, we have for any integer $\ell$
$$
R(\ell M!) \equiv r(\ell) \, (\mod 1).
$$
For $\ell \in \{1, 2\}$ we have
 \begin{align*}
    \abs{r(0) - r(\ell)} \leq \sum_{k=1}^{d} \abs{\varepsilon_k} (M!)^k \ell^k
      \leq \sum_{k=1}^{d} \frac{1}{20 md \cdot (2 M!)^d} (M!)^{d} 2^d
      \leq \frac{1}{20m}.
  \end{align*}

For $\ell \in \{1, 2, 3\}$ we have $\abs{r(\ell) + \beta_{\ell M!} -r(0)} \le 1/(10m)$, which implies that the three real numbers $r(\ell) + \beta_{\ell M!}$ belong to an interval of length $1/(5m) < 1/(3m)$. This relation is incompatible with the fact that $\{\lfloor P(\ell M!) + \eta_{\ell M!} \rfloor \colon 1 \le \ell \le 3\}$ takes three different values. 
\end{proof}

\subsection{Non uniformity modulo $m\ge2$ of smoothly perturbed polynomials}

The proof of the previous result makes a crucial use of the fact that we can find at least three different digits in base $m$. Indeed, the observation that for any sequence $(b_n)_n \in\{0, 1\}^{\N}$, there exists a sequence $(\varepsilon_n)_n$ tending to $0$ as quickly as we wish such that for any $n$ we have $\lfloor 2n + 1 + \varepsilon_n\rfloor \equiv b_n \,(\mod m)$ shows that Theorem \ref{mmorethan2} cannot be extended without modification to the case when $m=2$. Theorem \ref{mis2}  shows that the case $m=2$ can be treated if we add some regularity condition. The next easy lemma explains which regularity we require.

\begin{lemma}\label{monotonicity}
Let $m \ge 2$ be an integer and $x_1, x_2, x_3$ be a monotonic sequence of real numbers such that $\abs{x_3-x_1} <1$. The triplet
 of the residues modulo $m$ of $\lfloor x_1 \rfloor, \lfloor x_2 \rfloor$ and $ \lfloor x_3 \rfloor$ cannot be $(0, 1, 0)$ nor $(1, 0, 1)$.
\end{lemma}
\begin{proof}
We assume that the sequence $x_1, x_2, x_3$ is non-decreasing. We have
$$
\lfloor x_1 \rfloor \le \lfloor x_3 \rfloor \le \lfloor x_1 + 1 \rfloor = \lfloor x_1 \rfloor +1.
$$
Since $\lfloor x_1 \rfloor$ and $\lfloor x_3 \rfloor$ have the same residue modulo $m$, they are equal. We have $\lfloor x_1 \rfloor \le \lfloor x_2 \rfloor \le \lfloor x_3 \rfloor$, which implies $\lfloor x_1 \rfloor = \lfloor x_2 \rfloor$, a contradiction. The case when the sequence $x_1, x_2, x_3$ is non increasing is treated in a similar way.
\end{proof}

\begin{theorem}\label{mis2}
Let $m$ and $d$ be positive integers with $m \ge 2$. There exist a positive integer $N=N(m,d)$ and a block $B \in \{0, 1\}^N$ such that for any real polynomial $P$ of degree $d$ and any real function $\eta \in \mathcal{C}^{d+1}([1, N])$ such that
\begin{equation}\label{conditioneta}
 \forall t \in [1, N] \colon \abs{\eta(t)} \le 1/20 \; \text{ and }\; \eta^{(d+1)}(t) \neq 0,
\end{equation}
there exists $n \in [1, N]$ such that
\begin{equation}\label{notB2}
\lfloor P(n) + \eta(n)\rfloor \not\equiv b_n\, (\mod m).
\end{equation}
\end{theorem}\begin{proof}
We let $M = M(1/(20m), d)$ (where $M(.,.)$ was defined in Theorem~\ref{thm_pol_small_coeffs}) and $N \ge (40mdM!)^{(d+3)!}$ be integers satisfying  Theorem \ref{thm_pol_small_coeffs} (with $d=d$ and $\delta =1/(20m)$) and define the block $B$ to consist of $N$ integers almost all of them being equal to $0$, with the exception that, for $k= 1, 2, \ldots, 2d$ one has $b_{(2kM!)} = 1$.

We assume that there exists a polynomial $P$ for which (\ref{notB2}) does not hold; in particular, we have $\Card \{ n \in [1, N] \colon \lfloor P(n) + \eta(n)\rfloor \equiv 0 \,(\mod m)\} = N-2d$. We let $R(x) = P(x)/m$ and $\beta(x) = \eta(x)/m$. The discrepancy of the sequence $(R(n) + \beta_n)_{1 \le n\le N}$ is larger than $1/3$. We can thus apply Theorem \ref{thm_pol_small_coeffs} with $\delta = 1/(20m)$. Let us write $R(x) = \sum_{k=1, d} \alpha_k x^k$; for any $k$ there exist coprime rational integers $p_k$, $q_k$ with $1 \le q_k \le M$ and
$$
\abs{\alpha_k - \frac{p_k}{q_k}} \le N^{-\frac{(k+2)!}{2(d+2)!}} \le \frac{1}{20md\cdot (4dM!)^{d}},
$$
where the last inequality comes from the choice of $N$.

For $k \in [1, d]$ we let $\varepsilon_k = \alpha_k -p_k/q_k$ and $r(x) = \alpha_0 + \sum_{k=1}^d \varepsilon_k (M!)^k x^k$. Since $M!/q_k \in \N$, we have for any integer $\ell$
$$
R(\ell M!) \equiv r(\ell) \, (\mod 1).
$$

We define a function $f$ by the relation
$$
\forall t \in [1, 4d] \colon f(t) = m r(t) + \eta(tM!).
$$
For any $t \in [1, N]$, we have
$$
\abs{f(t)-f(0)} \le m\sum_{k=1}^{d} \abs{\varepsilon_k} (M!)^k \ell^k +2\times \frac{1}{20}
      \leq \sum_{k=1}^{d} \frac{1}{20 d \cdot (4dM!)^d} (M!)^{d} (4d)^d + \frac{1}{10}
      \leq \frac{3}{20}
$$
and so for any $t_1$ and $t_2$ one has $\abs{f(t_1)-f(t_2)} \le \frac{3}{10}$.\\

Since $r$ is a polynomial of degree $d$, we have $f^{(d+1)}(t) = (M!)^{d+1} \eta^{(d+1)}(tM!) $, which is different from $0$ by (\ref{conditioneta}). By repeated applications of Rolle's theorem, we find that $f'$ vanishes at most $d$ times on $[1, 4d]$: there exists at least an integer $\ell_0 \in [1, 4d]$ such that the sequence $f(\ell_0), f(\ell_0+1), f(\ell_0+2)$ is monotonic. 

By Lemma \ref{monotonicity}, the triplet $(\lfloor f(\ell_0) \rfloor, \lfloor f(\ell_0 + 1) \rfloor, \lfloor f(\ell_0 +2) \rfloor)$ taken modulo $m$ cannot be $(0, 1, 0)$ nor $(1, 0, 1)$.

We finally notice that, for any integer $\ell$, the difference between $f(\ell)$ and $P(\ell M!) + \eta(\ell m!)$ is a multiple of $m$; thus the triple $(\lfloor P(\ell_0M!) + \eta(\ell_0M!) \rfloor, \lfloor P((\ell_0+1)M!) + \eta((\ell_0+1)M!) \rfloor, \lfloor P((\ell_0+2)M!) + \eta((\ell_0+2)M!) \rfloor)$ taken modulo $m$ cannot be $(0, 1, 0)$ nor $(1, 0, 1)$, contrary to our assumption. This contradiction proves Theorem \ref{mis2}.

\end{proof}

\subsection{Proof of Theorem  \ref{thm_not_normal}}

Let $m \ge 2$ be an integer and $c>1$ a real number which is not an integer. We let $d=\lfloor c \rfloor$ and consider the number $N=N(m,d)$ and the block $B$ given by Theorem \ref{mis2}. For any positive real numbers $X$ and $t$, we consider the Taylor approximation of order $d$ of $(X+t)^c$, namely
$$
(X+t)^c = P_X(t) + \eta_X(t), \text{ where } P_X(t)=\sum_{k=0}^d \frac{c^{\underline{k}}X^{c-k}}{k!}t^k \text{ and } \mid\eta_X(t)\mid \le c t^{(d+1)}X^{\{c\}-1}.
$$

For any sufficiently large integer $X$, say $X \ge X_0$, we have $\mid\eta_X(t)\mid \le 1/20$ for any $t$ in $[1, N]$. Moreover, the $(d+1)$-th derivative of $\eta_X$ is the $(d+1)$-th derivative of $t \mapsto (X+t)^c$ and thus does not vanish. We can thus apply Theorem \ref{mis2}: there exists a block $B$ of length $N$ which does not occur in the sequence of the residues modulo $m$ of the sequence $(\lfloor n^c \rfloor)_{n \ge X_0}$. Let $U$ be in $\{0, 1, \cdots, m-1\}^{X_0}$; the word $UB$, concatenation of the words $U$ and $B$, never occurs in the sequence of the residues modulo $m$ of the sequence $(\lfloor n^c \rfloor)_{n \ge 0}$. This ends the proof of Theorem \ref{thm_not_normal}. \\

	We end this section by noticing that for $m\ge 3$ the more general Theorem \ref{mmorethan2} is sufficient for proving Theorem \ref{thm_not_normal}.

\section{Proof of Theorem~\ref{thm_complexity}}
\subsection{An upper bound for the complexity}
Assume that $L\geq 1$ (a block length) and $\varepsilon=1/(4L^2)$.
Write $\mathbf a_n=\lfloor n^c\rfloor\bmod m$.
By Taylor's theorem (consider the second derivative of $x^c$, which tends to $0$ like $x^{c-2}$) there exists a constant $C$, only depending on $c$, such that for $a\geq 4/(2-c)$ and $N\geq CL^a$ the following is satisfied:
There are reals $\alpha$, $\beta$ such that
\[0\leq n^c-(n\alpha+\beta) < \varepsilon\]
for $N\leq n<N+L$.
We also assume that $\alpha$ is irrational, which is no loss of generality.
This technical condition will be used later, when we apply the three gaps theorem.
The number of different factors in $\mathbf a$ of length $L$ occurring at positions $N<CL^a$ is trivially bounded by $CL^a$, which gives the first term of the maximum in the theorem.

It remains to consider start positions $N\geq CL^a$, where linear approximation of quality $\varepsilon$ can be applied.
Any block $(\mathbf a_N,\ldots,\mathbf a_{N+L-1})$ is obtained by starting from a block $b=(\lfloor N\alpha+\beta\rfloor\bmod m,\ldots,\lfloor (N+L-1)\alpha+\beta\rfloor\bmod m)$ and possibly modifying this sequence at indices $n$ such that $1-\varepsilon\leq \{n\alpha+\beta\}<1$.
This possible modification consists in adding $1$ modulo $m$.

We begin by estimating the number of factors of $\lfloor n\alpha+\beta\rfloor \bmod m$. Each such block corresponds to a finite Sturmian word by considering the sequence of differences $\lfloor (n+1)\alpha+\beta\rfloor - \lfloor n\alpha+\beta\rfloor$.
Note that such a sequence of differences corresponds to at most $m$ factors of the Beatty sequence.
This follows by taking the first element $\lfloor n_0\alpha+\beta\rfloor\bmod m$ of the considered factor into account and considering partial sums.
Using Mignosi~\cite{M1991} we can estimate the number of factors $b$ by $O(L^3)$, where here and in the following the implied constant may depend on $m$.

Consider the interval $I=[1-\varepsilon,1)$ and the set
\[A=\{n:N\leq n<N+L,\{n\alpha+\beta\}\in I+\mathbb Z\}.\]
We make use of the three gaps theorem (see, for example, the survey by Alessandri and Berthe~\cite{AB1998}, in particular the remark in section~4), which implies that there are at most three differences $a_2-a_1$ between consecutive elements of the set $B=\{n\in\mathbb N: \{n\alpha+\beta\}\in I+\mathbb Z\}$ and if three differences occur, the largest one is the sum of the smaller ones.

We distinguish between three cases.

\noindent (1) All gaps are $\geq L$. In this case, $\lvert A\rvert \leq 1$, so that we have to change the block $b$ at at most one position by adding $1$ modulo $m$, as noted above.
This gives a factor of $L+1$, which implies that this case contributes $O(L^{3+1})$ many cases.

\noindent (2) Exactly one gap is $<L$.
In this case $A$ is an arithmetic progression,
consisting of the elements $n_j=n_0+jd$ for some $n_0=\min A$ and $d\geq 1$, and $0\leq j<k$.
Set $x_j=\{n_j\alpha+\beta\}$ and $\delta=x_1-x_0$.

First, we want to show that $(x_0,\ldots,x_{k-1})$ is an arithmetic progression with difference $\delta$. Note that $\lvert \delta\rvert<\varepsilon<1/2$.
Suppose that we have already shown that $(x_0,\ldots,x_{j-1})$ is an arithmetic progression.
Clearly we have $x_j=x_{j-1}+\delta+r$ for some $r\in\mathbb Z$.
Suppose that $r\neq 0$. Then $\lvert x_j-x_{j-1}\rvert>1-\varepsilon$.
Since both $x_j$ and $x_{j-1}$ are elements of the interval $[1-\varepsilon,1)$, this is a contradiction to $\varepsilon<1/2$.

Next, we prove that the set $J=\{i<k: \lfloor n_i^c\rfloor>\lfloor n_i\alpha+\beta\rfloor\}$ is an interval.
To this end, note that $\lfloor n_i^c\rfloor >\lfloor n_i\alpha+\beta\rfloor$ if and only if $n_i^c-(n_i\alpha+\beta)\geq 1-\{n_i\alpha+\beta\}$, which is the case if and only if
$n_i^c-(n_i\alpha+\beta)+\{n_0\alpha+\beta\}+i\delta-1\geq 0$.
Note that the left hand side is a convex function of $i$, which implies the assertion.

In order to obtain the block 
$(\mathbf a_N,\ldots,\mathbf a_{N+L-1})$ from the block $b=(\lfloor N\alpha+\beta\rfloor\bmod m,\ldots,\lfloor (N+L-1)\alpha+\beta\rfloor\bmod m)$,
we modify $b$ at indices $n_j$ for $j\in J$, where $J$ is the interval obtained above.
These indices form an arithmetic progression in $[N,N+L-1]$, of which there are $O(L^3)$ many.
(Note that in fact $L^2\log L$ is sufficient.)
This implies a contribution of $O(L^{3+3})$ for this case.

\noindent (3) There exist two gaps $g_1<g_2<L$. We are going to show that this case cannot occur.
We first note that $g_1\alpha\not\in\mathbb Z$. Otherwise, the set $B$ would contain an arithmetic progression with difference $g_1$, therefore the gap $g_2$ would not occur, a contradiction.
It follows that $0<\lVert g_1\alpha\rVert<\varepsilon$.
Choose $n_1,n_2\in\mathbb N$ in such a way that $n_2-n_1=g_1$ and $n_1,n_2\in B$.
Consider the $g_1$ points $\{n_1\alpha+\beta\},\{(n_1+1)\alpha+\beta\},\ldots,\{(n_1+g_1-1)\alpha+\beta\}$.
These points dissect the torus into $g_1$ many intervals.
Therefore there is an interval $J=[x,y)$ in $\mathbb R$ of length $\geq 1/g_1\geq 1/L= 4L\varepsilon$ such that $n\alpha+\beta\not\in J+\mathbb Z$ for $n_1\leq n<n_2$.
Assume that $\{g_1\alpha\}<\varepsilon$, the case $\{g_1\alpha\}>1-\varepsilon$ being analogous.
Then 
\begin{align*}
\{n\alpha+\beta:n_1\leq n<n_1+2Lg_1\}
&=\bigcup_{0\leq k<2L}\{n\alpha+\beta:kn_1\leq n<n_1+(k+1)g_1\}
\\&\subseteq
\{n\alpha+\beta:n_1\leq n<n_1+g_1\}
+
\bigcup_{0\leq k<2L}\bigl(kg_1\alpha+\mathbb Z\bigr)
\\&\subseteq \mathbb R\setminus(J'+\mathbb Z),
\end{align*}
where $J'=[x+2L\varepsilon,y)$ has length $\geq 2\varepsilon$.
We now use the fact that $0\neq \lVert g_1\alpha\rVert<\varepsilon$ in order to shift the interval $J'$ over the interval $[1-\varepsilon,1)$ by using a multiple of $\alpha$.
Set $\delta=(1-2\varepsilon)-(x+2L\varepsilon)$ and assume that $n_0$ is such that $n_0\alpha\in \delta+[0,\varepsilon)+\mathbb Z$.
Then for all $n$ such that $n_1+n_0\leq n<n_1+n_02Lg_1$ we have
\begin{align*}
n\alpha+\beta&\in\mathbb R\setminus(J'+\mathbb Z)+n_0\alpha
\\&\subseteq [y,x+2L\varepsilon+1)+\delta+[0,\varepsilon)+\mathbb Z
\\&\subseteq [x+2L\varepsilon+2\varepsilon,x+2L\varepsilon+1)+(1-2\varepsilon)-(x+2L\varepsilon)+[0,\varepsilon)+\mathbb Z
\\&\subseteq [0,1-\varepsilon)+\mathbb Z.
\end{align*}
It follows that the sequence $(n\alpha+\beta)_{n\geq 0}$ does not visit $I+\mathbb Z$ for at least $2L$ many steps, which is a contradiction to the three gaps theorem: we proved the existence of a gap $\geq 2L$, but it is the sum of the smaller ones, therefore is at most $2L-1$.

Therefore case (3) does not occur, and the first part of the theorem is proved.

\subsection{A lower bound for the complexity}
We use Mignosi~\cite{M1991} again, this time we use the fact that there are at least $Ck^3$ Sturmian words of length $k$.
Let $a\in\{0,1\}^k$ be a subword of a Sturmian word.
There exist an irrational $\alpha_0<1$ and some $\beta_0$ and $n$ such that
$a_\ell=\lfloor (n+\ell+1)\alpha_0+\beta_0\rfloor-\lfloor (n+\ell)\alpha_0+\beta_0\rfloor$
for $0\leq\ell<k$.
Let $b\in\{0,\ldots,m-1\}^k$ be the sequence of partial sums modulo $m$;
moreover, let sequences $b^{(j)}$ be defined by $b^{(j)}(\ell)=b(\ell+j)\bmod m$.
Then one out of $b=b^{(0)},\ldots,b^{(m-1)}$ appears as a subword of $L(\beta_0)$, where $L(\beta)=(\lfloor n\alpha_0+\beta\rfloor\bmod m)_{n\geq 0}$.
By irrationality of $\alpha$ and the three gap theorem it is not difficult to show that there is some $B$ such that, for all $\beta$, every subword of $L(\beta)$ of length $B$ contains one subword taken from the set $\{b^{(0)},\ldots,b^{(m-1)}\}$.
(Sturmian words are \emph{uniformly recurrent}.)
We are going to show that $(\lfloor n^c\rfloor \bmod m)_{n\geq 0}$ contains a subword of $L(\beta)$ of length $B$, which establishes our claim.

It is elementary to show that there exists an $\varepsilon>0$ and some open interval $I$, such that the following holds:
for all $\beta$ and $n$ such that $n\alpha_0+\beta\in I+\mathbb Z$, we have
$\lVert (n+m)\alpha_0+\beta\rVert>\varepsilon$ for all $m<B$.

We use the denseness of $n\alpha_0+\beta\bmod 1$:
let $A$ be so large that for all $\beta$ and $n$ we have
$(n+\ell)\alpha_0+\beta\in I+\mathbb Z$ for some $\ell<A$.

Finally, let $x_0$ be so large such that $\frac 12f''(x)(A+B+1)^2<\varepsilon$ for $x\geq x_0$, where $f(x)=x^c$.
We approximate $x^c$ by a linear function $x\alpha+\beta$ at some point $x\geq x_0$ satisfying $\alpha=f'(x)\equiv \alpha_0\bmod m$.

We obtain some $\ell<A$ such that
$\lVert (\lceil x\rceil+\ell+m)\alpha+\beta\rVert>\varepsilon$ for all $m<B$.
By Taylor's formula it follows that $\lfloor (\lceil x\rceil+\ell+m)^c\rfloor =\lfloor (\lceil x\rceil +\ell+m)\alpha+\beta\rfloor$ for all $m<B$, which shows that there is a subword of $L(\beta)$ of length $B$ contained in $(\lfloor n^c\rfloor\bmod m)_{n\geq 0}$.

It follows that $\lfloor n^c\rfloor \bmod m$ has complexity at least $Ck^3$, which shows (using Allouche and Shallit~\cite[Corrolary~10.4.9]{AS2003}) that this sequence is not morphic. In particular, it is not an automatic sequence.

\section{Proof of Theorem~\ref{thm_sarnak}}

Theorem~\ref{thm_sarnak} is an easy corollary of the following result

\begin{theorem}\label{thm_daboussi-sarnak}
Suppose that $c>1$ is not an integer. Then we have for every multiplicative
function $f(n)$ with $|f(n)|\le 1$ and for every $\alpha\in \mathbb{Q}\setminus\mathbb{Z}$
\[
\sum_{n< N} f(n)\e\left({\alpha\lfloor n^c \rfloor}\right) = o(N), \qquad (N\to\infty).
\]
\end{theorem}

By the Daboussi-K\'atai criterion \cite{K1986} it is sufficient to prove 
\begin{equation}\label{eqn_toprove}
S := \sum_{n< N} \e\left( \alpha\left(\lfloor (pn)^c \rfloor - \lfloor (qn)^c \rfloor \right) \right) 
= o(N), \qquad (N\to\infty)
\end{equation}
for sufficiently large (and different) prime numbers $p,q$.

Actually it is important that we do not have to check (\ref{eqn_toprove}) for all pairs of (different) primes
$p,q$ since we have to exclude those cases where $(q/p)^c$ is rational. Fortunately this can
only occur for finitely many cases.

\begin{lemma}\label{lem41}
Suppose that $c>0$ is not an integer. Then there 
exists a constant $L> 0$ such that for all pairs of different primes 
$p,q > L$ we have
\[
(p/q)^c \not\in \mathbb{Q}.
\]
\end{lemma}

\begin{proof}
If there is at most one pair of different primes $p,q$ such that 
$(p/q)^c \in \mathbb{Q}$ then we set $L= 1$ or $L = \max\{p,q\}$.

Suppose next that there are two different pairs $(p_1,q_1)$, $(p_2,q_2)$ of primes with
$p_1 > q_1$, $p_2> q_2$
\[
(p_1/q_1)^c = r_1 \in \mathbb{Q} \quad\mbox{and}\quad (p_2/q_2)^c = r_2\in \mathbb{Q}
\]
and suppose that $(p_3,q_3)$ is another pair of different primes with
$p_3> q_3 > \max\{p_1,q_1,p_2,q_2\}$ such that
\[
(p_3/q_3)^c = r_3 \in \mathbb{Q}.
\]
Setting $\lambda_{11} = \log(p_1/q_1)$, $\lambda_{12} = \log(p_2/q_2)$, $\lambda_{13} = \log(p_3/q_3)$ and
$\lambda_{21} = \log r_1$, $\lambda_{22} = \log r_2$, $\lambda_{23} = \log r_3$ it follows that
\[
\frac{\lambda_{11}}{\lambda_{21}} =\frac{\lambda_{12}}{\lambda_{22}} =\frac{\lambda_{13}}{\lambda_{23}} = \frac 1c 
\]
or equivalently that the matrix
\[
M = \left( \begin{array}{ccc}  \lambda_{11} & \lambda_{12} & \lambda_{13} \\ \lambda_{21} & \lambda_{22} & \lambda_{23} 
\end{array} \right)
\]
has rank $1$.

By assumption it might be that one of $p_1,q_1$ coincides with
one of $p_2,q_2$ but not both. Hence, by unique factorization it follows that 
\[
\lambda_{11} = \log(p_1/q_1), \quad  \lambda_{12} = \log(p_2/q_2), \quad \lambda_{13} = \log(p_3/q_3)
\]
are linearly independent over the rationals. 

Furthermore we have the property that $c$ is irrational
or equivalently that $\lambda_{11}$ and $\lambda_{21}$ are linearly independent over
the rationals. Assuming the contrary it would have
\[
\left( \frac {p_1}{q_1} \right)^A = r_1^B
\]
for coprime integers $A,B$, that is, $c = A/B$. Recall that 
the primes $p_1$ and $q_1$ are different. 
Hence, $p_1$ has to appear on the right hand side, and due to the exponent $B$ it has
to appear with an integer multiple of $B$ as its multiplicity. However, due unique factorization
this multiplicity has to be $A$
which implies that $B=1$ and consequently that $c = A$ is an integer. But this is 
excluded by assumption.

Thus, by the Six Exponential Theorem by Lang~\cite{L1966} and Ramachandra \cite{Ra1967} implies that
the matrix $M$ has rank $2$. This leads to a contradiction and proves the lemma by
setting $L = \max\{p_1,q_1,p_2,q_2\}$.
\end{proof}

The next ingredient that we need is the following estimate for exponential sums.
\begin{lemma}\label{lem42}
Suppose that $c> 1$ is not an integer. Then we have uniformly for all real numbers $U$ with
$|U|\ge \eta$, where $\eta > 0$,  and $N\ge 1$
\[
\sum_{n\le N} e\left( U n^c \right) \ll |U| N^{1 - \frac{\| c\| }{2^c}}.
\]
\end{lemma}

\begin{proof}
The proof runs along the same lines as that of Proposition~\ref{lem_3}. 
\end{proof}

Now suppose that $p,q$ are different primes such that $(p/q)^c$ is irrational- Let 
$H$ be an arbitrary large number and observe first that we can slightly modify $S$
from (\ref{eqn_toprove}):
\begin{align*}
S &= \sum_{n< N} \e\left( \alpha\left(\lfloor (pn)^c \rfloor - \lfloor (qn)^c \rfloor \right) \right) \\
&= \sum_{n< N} e\left(\alpha \lfloor (pn)^c \rfloor - \alpha  \lfloor (qn)^c \rfloor\right) \\
&= \sum_{n< N} e\left(\alpha \left((pn)^c -(qn)^c \right) - \alpha \left( \{ (pn)^c \} - \{ (qn)^c \} \right) \right) \\
&= \sum_{0\le k_1,k_2 < H} 
\sum_{n< N,\, \{ (pn)^c \} \in [\frac {k_1}H, \frac{k_1+1}H),\,  \{ (qn)^c \} \in [\frac {k_2}H, \frac{k_2+1}H) } 
e\left(\alpha \left((pn)^c -(qn)^c \right) - \alpha \left( \frac {k_1}H - \frac{k_2} H \right) \right) 
+ O\left( \frac NH \right) 
\end{align*}
It is, thus, sufficient to study the sums
\begin{align*}
S_{k_1,k_2} &:= \sum_{n< N,\, \{ (pn)^c \} \in [\frac {k_1}H, \frac{k_1+1}H),\,  \{ (qn)^c \} \in [\frac {k_2}H, \frac{k_2+1}H) } 
e\left(\alpha \left((pn)^c -(qn)^c \right) \right) \\
&= \sum_{n< N} 
e\left(\alpha \left((pn)^c -(qn)^c \right) \right)
\chi_{1/H}\left( (pn)^c - \frac {k_1}H \right) 
\chi_{1/H}\left( (qn)^c - \frac {k_2}H \right) 
\end{align*}
For this purpose we approximate the indicator function $\chi_{1/H}$ with the help of a Lemma due to Vaaler (see \cite[Theorem A.6]{GK1991})
and obtain
\begin{align*}
&\left| \chi_{1/H}(x) \chi_{1/H}(y) - \sum_{|h_1|, |h_2| \le H^3 } a_{h_1}(H^{-1},H^3) a_{h_2}(H^{-1},H^3) e(h_1x + h_2 y) \right| \\
& \ll \sum_{ |h| \le H^3 } b_{h}(H^{-1}, H^3) \left( e(hx) + e(hy) \right),
\end{align*}
where 
\[
a_0(H^{-1},H^3) = \frac 1H, \quad, | a_h(H^{-1},H^3) | \le \min\{ \frac 1H, \frac 1{\pi|h|} \}, \quad 
| b_h(H^{-1},H^3) | \le \frac 1{H^3 +1}.
\]
Thus, $S_{k_1,k_2}$ can be estimated by
\begin{align*} 
|S_{k_1,k_2}| & \ll \sum_{|h_1|, |h_2| \le H^3 } |a_{h_1}(H^{-1},H^3) a_{h_2}(H^{-1},H^3)| 
\left| \sum_{n< N} e\left(  \left( \left( \alpha + h_1 \right) p^c + 
\left( -\alpha + h_2 \right) q^c \right) n^c \right)  \right|\\
&+ \sum_{ |h| \le H^3 } |b_{h}(H^{-1}, H^3)| \left| \sum_{n< N} e\left( h p^c n^c \right) \right| 
+ \sum_{ |h| \le H^3 } |b_{h}(H^{-1}, H^3)| \left| \sum_{n< N} e\left( h q^c n^c \right) \right| \\
&\ll c_1(H) N^{1 - \frac{\|c\|}{2^c} } + \frac N{H^3} + c_2(H) N^{1 - \frac{\|c\|}{2^c} },
\end{align*}
where $c_1(H)$ and $c_2(H)$ are constants depending on $H$ and where 
we have used Lemma~\ref{lem42} and the property that (by assumption) 
\[
\min_{|h_1|, |h_2| \le H^3 } \left| \left( \alpha + h_1 \right) p^c + \left( -\alpha + h_2 \right) q^c \right| > 0.
\]
Summing up and letting $N\to\infty$ it follows 
\[
\limsup_{N\to\infty} \frac{ |S| } N \ll \frac 1H + H^2 \frac 1{H^3} \ll \frac 1H.
\]
Since $H$ can be chosen arbitrarily large it finally follows that $S = o(N)$ as $N\to\infty$.
This completes the proof of Theorem \ref{thm_daboussi-sarnak}
and thus that of Theorem~\ref{thm_sarnak}.

\bibliographystyle{siam}

\end{document}